\documentclass[12pt]{amsart}
 
\usepackage{amsmath, amsthm, amssymb, graphicx, verbatim, hyperref}
\usepackage[shortalphabetic]{amsrefs}
\usepackage{color}
\usepackage{stackrel}

\usepackage[margin=3cm]{geometry}

\newtheorem{thm}{Theorem}[section]

\newtheorem*{thm*}{Theorem}
\newtheorem{prop}[thm]{Proposition}
\newtheorem{lem}[thm]{Lemma}
\newtheorem{cor}[thm]{Corollary}

\theoremstyle{remark}
\newtheorem*{rmk}{Remark}

\newcommand{\IP}[2]{\left<#1,#2\right>}

\allowdisplaybreaks[2]

\newcommand{\BG}{\mathbb{G}}
\newcommand{\SP}{\mathcal{P}}
\newcommand{\R}{\mathbb{R}}

\newcommand{\N}{\mathbb{N}}
\newcommand{\SK}{\mathcal{K}}

\renewcommand{\S}{\mathbb{S}}
\newcommand{\vn}[1]{\lVert#1\rVert}

\newcommand{\kav}{\overline{k}}

\begin{document}

\title{Closed ideal planar curves}
\author[B. Andrews, J. McCoy, G. Wheeler, V.-M. Wheeler]{Ben Andrews\and James McCoy\and Glen Wheeler\and Valentina-Mira Wheeler*}

\thanks{Financial support from Discovery Project DP150100375 of the Australian Research Council is gratefully acknowledged.\\ \hspace*{0.8em} *: Corresponding author.}
\address{Applied and Nonlinear Analysis Group, Mathematical Sciences Institute, College of Science, Australian National University, Canberra, ACT, Australia}
\email{ben.andrews@anu.edu.au}
\address{Priority research centre for Computer-Assisted Research Mathematics
and its Applications, School of Mathematical \& Physical Sciences, Mathematics
Building - V122, University of Newcastle, University Drive, Callaghan NSW 2308
Australia}
\email{james.mccoy@newcastle.edu.au}
\address{Institute for Mathematics and its Applications, School of Mathematics and Applied Statistics, University of Wollongong, Northfields Ave, Wollongong, NSW 2500, Australia}
\email{glenw@uow.edu.au\text{ and }vwheeler@uow.edu.au}

\begin{abstract}

In this paper we use a gradient flow to deform closed planar curves to curves
with least variation of geodesic curvature in the $L^2$ sense.
Given a smooth initial curve we show that the solution to the flow exists for
all time and, provided the length of the evolving
curve remains bounded, smoothly converges to a multiply-covered
circle. Moreover, we show that curves in any homotopy class with
initially small $L^3\vn{k_s}_2^2$ enjoy a uniform length bound under the
flow, yielding the convergence result in these cases.
\end{abstract}
\maketitle

\section{Overview}

Let us define $E[\gamma]$ for a smooth closed planar curve $\gamma:\S\rightarrow\R^2$
by
\[
	E[\gamma] = \frac12\int_\gamma k_s^2\,ds\,.
\]
Here, $k_s$ is the first arclength derivative of curvature.
We note that applications of the energy $E$ appear in computer aided design \cite{application}.\footnote{We thank Yann Bernard for providing us with this reference.}
Critical points for $E$ are termed \emph{planar ideal curves}.
These are the model one-dimensional case of \emph{ideal submanifolds}, and
their study is a preliminary proof-of-concept step in a larger program.
We explain this broader perspective in Section \ref{Sbroad}, before moving into
the details of the case we study here in Section \ref{Splanar}.

Our primary interest is in the $L^2$-gradient flow for the functional $E$.
In Section \ref{Splanar} we calculate the first variation of $E$, giving the resultant Euler-Lagrange operator and gradient
flow.
Next, Section \ref{Sshort} describes local existence for the flow.
This is done by a standard method.
Section \ref{Srigid} concerns the equilibrium set for the flow, proving that it
consists only of multiply-covered circles (we call these $\omega$-circles) and
their rigid images.
We move on to the analysis of the global behaviour of the flow in Section
\ref{Sglobal}, which identifies \emph{length} as the quantity that influences
asymptotic behaviour of the flow.
The main result of Section \ref{Sglobal} is global existence for arbitrary
smooth initial data.
This behaviour is reminiscent of the (unconstrained) elastic flow \cite{DKS},
and is in stark contrast with the curve diffusion flow \cite{Wcdf} and its
higher-order relatives \cite{WP}.

In Section \ref{Sstability}, we study the behaviour of the flow under a condition on the scale-invariant counterpart of $E$, namely (here $L$ is used to denote the length of $\gamma$)
\[
(L^3E)[\gamma]\,.
\]
While the flow exists for all time regardless of the initial energy,
convergence is not straightforward.
Since the rigidity result of Section \ref{Srigid} classifies all equilibria as
circles, clearly initial data in the regular homotopy class of a lemniscate can
not converge.
The key issue is in obtaining uniform upper and lower bounds for the length of the evolving family of curves.
When length is straightforward to control, such as in the case of evolving
families of curves with free boundary on parallel lines \cite{para1}\footnote{See \cite{para2} for the published conference paper.},
convergence can be obtained by a more standard, direct argument.
In Section \ref{Sstability} we prove that smallness of $(L^3E)[\gamma_0]$
allows us to estimate the scale-invariant energy by the square of the
Euler-Lagrange operator in $L^2$.
This allows us to first prove monotonicity of the scale-invariant energy, and
eventually obtain uniform upper and lower bounds on length.

When full convergence of a gradient flow is difficult, adaptation of the classical \L ojasiewicz-Simon
gradient inequality is a powerful strategy.
Relying on an observation due to Chill \cite{Chill}, this was successfully completed in \cite{DPS} for the elastic flow.
Here, we are able to avoid adapting this framework through further careful
analysis of the Euler-Lagrange operator.
In particular, we show that if the energy is small, it must decay exponentially fast.
This is then enough to enact a standard argument to obtain full convergence of
the flow.
Apart from yielding exponential convergence, this further improves the
resultant convergence statement by removing the need for a family of
reparametrisations (compare with the convergence results in \cite{DPS}).  This
is reminiscent of, for example, Huisken's original convergence result for the
mean curvature flow \cite{huisken}.  This argument concludes Section
\ref{Sstability}.

We summarise the main results, Theorem \ref{TMglobal} and Theorem \ref{TMmain}, in the following statement.

\begin{thm}
\label{TMoverview}
Let $\gamma:\S\times[0,T)\rightarrow\R^2$ be the steepest descent $L^2$-gradient flow for the functional
\[
E[\gamma] = \frac12\int_\gamma k_s^2\,ds
\,,
\]
where $\gamma(\cdot,0) = \gamma_0(\cdot)$ is smooth
and $T$ is maximal, $T\in(0,\infty]$.
Then:
\begin{enumerate}
\item[(a)] The maximal time of existence is infinite ($T=\infty$); and
\item[(b)] If the length $L[\gamma_t]$ is uniformly bounded along the flow, then
$\gamma$ converges exponentially fast in the $C^\infty$-topology to a standard round
$\omega$-circle, where $\omega = \frac1{2\pi}\int_{\gamma_0}k_0\,ds_0$.
\end{enumerate}
Furthermore, there exists a universal constant $\varepsilon_2>0$ such that if
$\gamma_0:\S\rightarrow\R^2$ satisfies 
\begin{equation*}
	(L^3E)[\gamma_0] < \varepsilon_2\,,
\end{equation*}
then length is uniformly bounded along the flow, and the convergence statement from (b) above holds.
\end{thm}

\section{Ideal submanifolds}
\label{Sbroad}

It is a classical pursuit to find geometric shapes that exhibit desirable properties.
Let us take $X:M^n\rightarrow\R^{n+m}$ to be a smooth closed immersed submanifold of Euclidean space.
Often, in the case of submanifold theory, desirable properties are determined by the \emph{curvature} of the submanifold.
For example, minimal submanifolds have vanishing mean curvature vector, and CMC submanifolds have parallel mean curvature vector.
In the case of hypersurfaces ($m=1$), this means that the mean curvature scalar is constant.

Minimal hypersurfaces have a variational characterisation in terms of the $L^2$-gradient of the area functional
\[
	A[X] = \int_{M^n} d\mu
\]
where $d\mu$ is the area measure induced by $X$.
For smooth hypersurfaces with constant mean curvature, there are several variational approaches.
The first and most classical is to minimise $A(X)$ subject to a constraint on enclosed volume.
In terms of gradient flows, this perspective gives rise to the
\emph{volume-preserving mean curvature flow} (VPMCF), see
\cite{huiskenvpmcf}.\footnote{Another natural approach is to fix area, see
\cite{mccoyapmcf}.} 

Another approach is to build the volume- or area-preservation property into the
definition of the underlying Hilbert space, where the gradient is being taken.
From this perspective, we consider not the $L^2$-gradient of $A$ with
constraint but instead the $H^{-1}$-gradient of $A$.
This gives rise to the \emph{surface diffusion flow} (SDF), see \cite{wh6}.
The set of smooth equilibria set for both flows consists of CMC hypersurfaces, but the operators are very different.
For (VPMCF), the velocity is
\[
	H - \frac{1}{A[X]}\int_{M^n}H\,d\mu
\]
whereas for (SDF) the velocity is
\[
	\Delta H\,.
\]
Both have zero average, which is why both preserve enclosed volume. However one
is non-local and second order, and the other is fourth order.

In this paper we propose a third option.
The motivation for this choice is as follows.
For (VPMCF), we have a minimisation problem with a constraint.
This results in a non-local operator.
The (SDF) gives a local operator, but the variational problem is not in $L^2$; rather, it is in $H^{-1}$.
Our proposal is to consider the functional
\[
	E[X] = \frac12\int_{M^n} |\nabla H|^2\,d\mu
\]
and minimise $E$ in $L^2$.
This is a variational problem in $L^2$ and the operator is local.
Of course, nothing is free, and the resultant operator is now of sixth order.
However we feel that, intuitively at least, this is at least as good a situation as (VPMCF) and (SDF).

To test this intuition we need to check a few essential points:
\begin{itemize}
\item That the set of smooth equilibria consists only of CMC hypersurfaces; and
\item That the $L^2$-gradient flow around CMC hypersurfaces is stable.
\end{itemize}
The main results of the present paper are the confirmation of both of these points in the simplest case of $n=m=1$.

\section{Planar ideal curves}
\label{Splanar}

Let us now give the details of the mathematical setting.
Suppose $\gamma_0:\S\rightarrow\R^2$ is a circle immersed regularly in the plane, and consider the energy
\[
E[\gamma_0] = \frac12\int_\gamma k_s^2\, ds\,,
\]
where $s$ denotes the Euclidean arc-length, $k = \IP{\gamma_{ss}}{\nu}$ is the
curvature, $\tau = (\tau_1, \tau_2) = \gamma_s$ is the tangent vector and $\nu = (-\tau_2,
\tau_1)$ the unit normal vector along $\gamma$.
Our convention here is that the normal vector points into the interior of
$\gamma$.

Consider now a one-parameter family of curves $\gamma:\S\times[0,T)\rightarrow\R^2$ evolving with a purely normal velocity
\footnote{This procedure can be carried out analogously with a
	tangential component to the velocity; this will carry through and then
cancel out at appropriate moments in the derivation. This is because $E$ does not depend on parametrisation, and tangential terms correspond only to reparametrisation.
}
\[
\partial_t\gamma = V\nu
\,.
\]
The commutator of the time and arc-length derivatives is given by
\[
[\partial_t,\partial_s] = kV\partial_s\,,
\]
and the measure $ds$ evolves by
\[
\partial_t\,ds = -kV\,ds\,.
\]
Since $[\partial_t,\partial_s]\gamma = kV\tau$, we find
\[
\partial_t\tau = \partial_s(V\nu) + kV\tau = V_s\nu\,.
\]
By the orthonormality of $\{\tau,\nu\}$, $\partial_t\nu = -V_s\tau$.
Similarly, as $[\partial_t,\partial_s]\tau = k^2V\nu$, we find
\[
\partial_t\kappa = \partial_s(V_s\nu) + k^2V\nu
 = (V_{ss}+k^2V)\nu - V_sk\tau\,,
\]
so that
\[
\partial_tk = \partial_t\IP{\kappa}{\nu}
 = V_{ss}+k^2V\,.
\]
Now $[\partial_t,\partial_s]k = Vkk_s$, so
\[
\partial_tk_s = \partial_s(V_{ss} + k^2V) + Vkk_s
 = V_{s^3} + V_sk^2 + 3Vkk_s\,.
\]
The evolution of the functional $E$ can now be calculated:
\begin{align*}
\frac{d}{dt}\frac12\int_\gamma k_s^2\,ds
 &= \int_\gamma k_s(V_{s^3} + V_sk^2 + 3Vkk_s)\,ds - \int_\gamma Vkk_s^2\,ds
\\
 &= \int_\gamma V\bigg[
                    -k_{s^4} - k_{ss}k^2 - 2kk_s^2 + 3kk_s^2 - \frac12kk_s^2
                  \bigg]\,ds
\\
 &= \int_\gamma V\bigg[
                    -k_{s^4} - k_{ss}k^2 + \frac12k_s^2k
                  \bigg]\,ds
\,.
\end{align*}
Above we have used the notation $V_{s^3} = V_{sss}$ and $k_{s^4} = k_{ssss}$.
Now for the flow $\gamma$ to be the steepest descent gradient flow of $E$ in $L^2$, we must have $E' = -\vn{\gamma_t}_2^2$, that is, we require
\[
V = k_{s^4} + k_{ss}k^2 - \frac12k_s^2k\,.
\]
We therefore have the flow:
\begin{equation}
\label{EQflow}
\partial_t\gamma = \bigg(k_{s^4} + k_{ss}k^2 - \frac12k_s^2k\bigg)\nu\,.
\end{equation}


\section{Local existence and uniqueness}
\label{Sshort}

By writing the solution locally in time as a graph over the
initial data and using classical PDE theory, we obtain the following local
well-posedness result.
This procedure is carried out in detail in \cite{CharlieThesis} for a more
general class of equations than we consider here.

\begin{thm}
\label{TMste}
Let $\gamma_0:\S\rightarrow\R^2$ be a closed immersed curve of class $C^{6,\alpha}$.
There exists a unique smooth maximal family
$\gamma:\S\times[0,T)\rightarrow\R^2$, $T\in(0,\infty]$,  of immersed curves
such that $\gamma(s,0) = \gamma_0(s)$ and
\[
	\partial_t\gamma = 	\Big(
				k_{s^4} + k_{ss}k^2 - \frac12k_s^2k
				\Big)\nu
				\,.
\]
Furthermore, if $T<\infty$, then the quantity
\[
Q[\gamma_t] = L[\gamma_t] + \int_\gamma k_{s^5}^2\,ds
\]
is unbounded as $t\rightarrow T$.
\end{thm}

\begin{rmk}
The regularity condition and blowup criterion of Theorem \ref{TMste} are not optimal.
One natural hypothesis would be that the flow exists uniquely for initial data of class $W^{3,2}$.
Then, the energy controls automatically the high derivative term in the $W^{3,2}$-norm, whereas the lowest order term is controlled by the length.
We identify in Sections \ref{Sglobal} and \ref{Sstability} estimates on length as the critical ingredient for convergence of the flow.
\end{rmk}

\begin{rmk}
If the quantity $Q[\gamma_t]$ remains uniformly bounded and $T<\infty$ is maximal, then by
the standard Sobolev inequality $\gamma$ is uniformly bounded in
$C^{6,\alpha}$, and we may assert that the flow converges as $t\rightarrow T$
to a $C^{6,\alpha}$ limiting curve.
This then allows us to apply again the short time existence theorem,
contradicting the maximality of $T$.
\end{rmk}


\section{Rigidity for closed ideal curves}
\label{Srigid}

Let $\gamma$ be a closed curve satisfying
\[
	\SK[\gamma] := k_{s^4} + k_{ss}k^2 - \frac12k_s^2k = 0
	\,,
\]
that is, a stationary solution to the $L^2$-gradient flow of $E$. Recall that such curves are called \emph{ideal}.
In this section we prove:

\begin{thm}
\label{TM1}
Suppose $\gamma:\S\rightarrow\R^2$ is a smooth closed ideal curve.
Then $\gamma(\S) = \S_r(x)$, that is, $\gamma$ is a standard round $\omega$-circle.
\end{thm}

\begin{proof}
{\bf Integrating the equation.}
Since $\SK[\gamma] = 0$ we find by a standard argument that $\gamma$ is analytic and
\begin{align*}
(k_{s^3}^2)_s &+ 2k_{s^3}k_{ss}k^2 - k_{s^3}k_s^2k
= (k_{s^3}^2 + k_{ss}^2k^2)_s - 2k_{ss}^2k_sk - k_{s^3}k_s^2k
\\
&= (k_{s^3}^2 + k_{ss}^2k^2 - k_{ss}k_s^2k)_s - 2k_{ss}^2k_sk + 2k_{ss}^2k_sk + k_{ss}k_s^3
\\
&= 0
\end{align*}
which implies
\[
	\Big(
	k_{s^3}^2 + k_{ss}^2k^2 + \frac14k_s^4 - k_{ss}k_s^2k
	\Big)_s = 0
	\,.
\]
Therefore there exists a $C\in\R$ such that for each $s\in[0,L]$,
\[
	Q(s) =
	k_{s^3}^2 + k_{ss}^2k^2 + \frac14k_s^4 - k_{ss}k_s^2k
	= C\,.
\]
Let $s_0\in[0,L]$ be a point where $k_s(s_0) = 0$ (note that $k$ is a smooth periodic function).
Then
\begin{equation}
\label{EQ1}
	Q(s_0) =
	(k_{s^3}^2 + k_{ss}^2k^2)(s_0) = C \ge 0\,.
\end{equation}
We have two cases.

{\bf Case 1: $C=0$.}
In this case integration yields
\[
\int_\gamma \Big[ k_{s^3}^2 + k_{ss}^2k^2 + \frac{7}{12} k_s^4 \Big]\,ds = 0\,.
\]
We conclude that $k_s$ is constant, which, together with the closedness of $\gamma$, implies the result.

{\bf Case 2: $C>0$.}
Consider the rescaling $\eta = \rho\gamma$.
We calculate on $\eta$
\[
	Q(s) =
	k_{s^3}^2 + k_{ss}^2k^2 + \frac14k_s^4 - k_{ss}k_s^2k
	= C\rho^{-8}\,.
\]
Choosing $\rho = C^\frac18$ we find
\begin{equation}
\label{EQ2}
	Q(s) =
	k_{s^3}^2 + k_{ss}^2k^2 + \frac14k_s^4 - k_{ss}k_s^2k
 = 1\,.
\end{equation}
Now we write $Q(s) = M^2(s) + N^2(s)$ where
\[
M(s) = k_{s^3}
\quad\text{and}\quad
N(s) = k_{ss}k - \frac12k_s^2\,.
\]
Equation \eqref{EQ2} implies that there exists a $\phi:[0,L]\rightarrow\R$ such that
\[
M(s) = \cos\phi(s)
\quad\text{and}\quad
N(s) = \sin\phi(s)
\,.
\]
Since
\[
N_s = \phi_s \cos\phi = k_{s^3}k = kM
\]
and
\[
M_s = k_{s^4} = -\phi_s\sin\phi = -\phi_sN
\]
we have by analyticity $\phi_s = k$.

Since the integral of $k$ is $2\omega\pi$ we have
\[
\phi(s+nL) = \phi(s) + 2\omega n \pi\,,
\]
where $\omega$ is the winding number of $\gamma$.
Let $\tau(s) = (x_s,y_s)$ where $\gamma = (x,y)$. Then $x_s^2 + y_s^2 = 1$ and again we find a function
$\theta:[0,L]\rightarrow\R$ such that
\[
\tau(s) = (\cos\theta(s), \sin\theta(s))\,.
\]
This function $\theta$ is (up to translation) the standard notion of tangential angle.
As $\tau_s = k\nu = \theta_s(s) (-\sin\theta,\cos\theta)$ and $\nu(s) =
(-\sin\theta,\cos\theta)$, we must also have that $\theta_s = k$ and so
\[
\theta(s) = \phi(s) + \theta_0
\]
for some $\theta_0\in\R$.
This implies that
\begin{align*}
\tau(s) &= (\cos(\phi(s) + \theta_0),\ \sin(\phi(s) + \theta_0))
\\
 &= (\cos\phi(s)\cos\theta_0 - \sin\phi(s)\sin\theta_0,\
    \sin\phi(s)\cos\theta_0 + \cos\phi(s)\sin\theta_0)
\,.
\end{align*}
In particular
\[
y_s = N(s)\cos\theta_0 + M(s)\sin\theta_0\,.
\]
By closedness, we have
\[
0 = \int_\gamma y_s\,ds
 = \cos\theta_0\int_\gamma N(s)\,ds + \sin\theta_0\int_\gamma M(s)\,ds\,.
\]
Now
\[
\int_\gamma M(s)\,ds = \int_\gamma k_{s^3}\,ds = 0
\]
and
\[
\int_\gamma N(s)\,ds = -\frac32\int_\gamma k_s^2\,ds\,.
\]
Therefore we find
\[
\cos\theta_0\int_\gamma k_s^2\,ds = 0\,.
\]
Either $k_s = 0$ and we again conclude the result, or $\cos\theta_0 = 0$.
In the latter case, we note that $\sin\theta_0 \ne 0$ and calculate
\[
0 = \int_\gamma x_s\,ds = \sin\theta_0 \frac32\int_\gamma k_s^2\,ds\,.
\]
This again implies $k_s = 0$ and so we are finished.

\end{proof}

\section{Global existence for arbitrary initial data}
\label{Sglobal}

Along an ideal curve flow, the functional $\vn{k_s}_2^2$ is monotone decreasing.
Inflating the curve by scaling $\gamma \mapsto \rho\gamma$ for
$\rho\rightarrow\infty$ decreases the energy to zero, regardless of $\gamma$.
This is why the flow tends to enlarge the initial data.
In fact, the flow may enlarge the initial data without end.
We prove in this section that the only kind of blowup that can occur along an
ideal curve flow is that length becomes unbounded, and that furthermore this
can not happen in finite time; implying that the ideal curve flow with smooth
initial data exists for all time.

\begin{thm}  
\label{TMglobal} 
Let $\gamma_0:\S\rightarrow\R^2$ be a smooth immersed curve,
and $\gamma:\S\times[0,T)\rightarrow\R^2$ be the ideal curve flow with $\gamma_0$ as initial data.
Then $T=\infty$, and if there exists an $L_0\in\R$ such that $L(t)\le L_0$ for all $t\in[0,T)$, 
$\gamma$ converges exponentially fast in the $C^\infty$-topology to a standard round $\omega$-circle.
\end{thm}

In this section we focus on the $T=\infty$ part of Theorem \ref{TMglobal},
leaving the convergence result for Section \ref{Sstability}.
There the ideal curve flow is studied under a condition on 
\begin{equation}
\label{smallenergypreview}
\frac12L^3(t)\vn{k_s}_2^2(t) = (L^3E)[\gamma_t]
\,.
\end{equation}
We prove that there exists an absolute constant $\varepsilon_2>0$ such that
$(L^3E)[\gamma_t] < \varepsilon_2$ implies convergence to an $\omega$-circle
(Theorem \ref{TMmain}).

In the present section, we wish to work without any condition on $L^3E$ at $t=0$.
However note that, since $E[\gamma]$ is decreasing, if at any time $L^3(t_0) <
\varepsilon_2/E[\gamma_0]$, the smallness condition would be satisfied at $t_0$,
and the results of the next section (and in particular Theorem \ref{TMmain})
would apply.
Therefore we assume that this is not the case for the remainder of this section, that is, that the estimate
\begin{equation}
\label{standingass}
L^3(t) \ge \frac{\varepsilon_2}{E[\gamma_0]}
\end{equation}
holds.

The local existence result Theorem \ref{TMste} states that if $T<\infty$, then
at least one of length or $\vn{k_{s^5}}_2^2$ become unbounded in finite time.
We will show below that neither one of these possibilities can occur.

We begin by noting the following scale-invariant curvature estimate.

\begin{lem}
\label{kest}
For any immersed curve $\gamma:\S\rightarrow\R^2$ we have the estimate
\[
L\vn{k}_\infty \le \sqrt{L^3\vn{k_s}_2^2} + 2\omega\pi
\,.
\]
Here $\omega$ is the winding number of $\gamma$.
\end{lem}
\begin{proof}
We calculate
\[
k = k - \kav + \kav \le \int_\gamma |k_s|\,ds + \frac{2\omega\pi}{L}
\,.
\]
Taking a supremum and using the H\"older inequality, we find
\[
\vn{k}_\infty \le \frac1L\bigg(\sqrt{L^3\vn{k_s}_2^2} + 2\omega\pi\bigg)
\,.
\]
\end{proof}

\begin{rmk}
Note that we do not need the hypothesis \eqref{standingass} for Lemma \ref{kest} to hold.
This will be useful in the next section.
\end{rmk}

Now we show that length can not become unbounded in finite time.

\begin{lem}
\label{lemlength}
Along any ideal curve flow $\gamma:\S\times[0,T)\rightarrow\R^2$ we have
\[
L[\gamma_t] \le \Big(L[\gamma_0]\exp\big(E[\gamma_0]\big)\Big) e^{c_0t}\,,
\]
where $c_0 = c_0(\omega,E[\gamma_0])$.
\end{lem}
\begin{proof}
We calculate and estimate using Lemma \ref{kest}
\begin{align*}
\frac{d}{dt}L
 &= -\int_\gamma k{\mathcal K}\,ds\leq \vn{k}_2\vn{\mathcal K}_2
\\
&\le \bigg(
	\frac{2\omega\pi}{\sqrt{L}}+\sqrt2 E^{\frac12}[\gamma_0]L
	\bigg)\vn{\mathcal K}_2
\,.
\end{align*}
Now, \eqref{standingass} implies
\[
L^{-1} \le \bigg(\frac{4E[\gamma_0]}{\varepsilon_0}\bigg)^\frac{1}3 := c_L\,.
\]
Therefore
\begin{align*}
\frac{d}{dt}\log L
	&\le L^{-\frac12}\bigg(
	\frac{2\omega\pi}{L}+\sqrt2 E^{\frac12}[\gamma_0]
	\bigg)\vn{\mathcal K}_2
\\
	&\le \sqrt{c_L}(2\omega\pi c_L + \sqrt2 E^{\frac12}[\gamma_0])\vn{\mathcal K}_2
\\
	&\le \frac{c_L}{4}(\sqrt{2}\omega\pi c_L + E^{\frac12}[\gamma_0])^2
	 + 2\vn{\mathcal K}_2^2
\,.
\end{align*}
Let us set $c_0 := \frac{c_L}{4}(\sqrt{2}\omega\pi c_L + E^{\frac12}[\gamma_0])^2$.
Integrating and using 
\begin{equation}
\label{gradflowstruct}
-2\int_0^t \vn{\SK}_2^2\,d\tau = \int_0^t \frac{d}{d\tau}\vn{k_s}_2^2\,d\tau
 = \vn{k_s}_2^2(t) - \vn{k_s}_2^2(0)
\,,
\end{equation}
we find
\[
\log L[\gamma_t] \le \log L[\gamma_0] + c_0t + E[\gamma_0]\,,
\]
or
\[
L[\gamma_t] \le \Big(L[\gamma_0]\exp\big(E[\gamma_0]\big)\Big) e^{c_0t}\,.
\]
\end{proof}

Let us now prove global existence, that is, $T=\infty$.
Although we only need to show that $\vn{k_{s^5}}_2^2 \le C(t_0)$ on $[0,t_0]$,
$t_0<\infty$, what we shall actually prove is that \emph{all} derivatives of
curvature are uniformly bounded on any bounded time interval.

\begin{thm}
\label{curvatureestimates}
Along any ideal curve flow $\gamma:\S\times[0,T)\rightarrow\R^2$
we have for all $l\in\N$,
\[
\vn{k_{s^l}}_\infty \le C(l)
\,,
\]
where $C(l)$ a constant depending only on $l$, $T$, $\omega$ and $E[\gamma_0]$.
\end{thm}
\begin{proof}
We first need to identify the structure of the evolution for $k_{s^l}$.
Recall the commutator equation
\[
[\partial_t,\partial_s] = k\SK\partial_s
\,.
\]
Recall also the base case,
\[
\partial_t k
 = \SK_{ss} + k^2\SK
\,.
\]
We calculate the additional terms that arise when moving up an order:
\[
\partial_t k_{s^{l+1}} = [\partial_t,\partial_s]k_{s^{l}} + \partial_s(\partial_tk_{s^l})
 = k\SK k_{s^{l+1}} + \partial_s(\partial_tk_{s^l})
\,.
\]
Now
\[
\SK = k_{s^4} + P^2_3(k)\,,
\]
where we use $P_i^j(u)$ to mean a linear combination of terms each consisting
of $j$ arc-length derivatives on $i$ copies of $u$.
Thus we have
\begin{align*}
\partial_t k
 &= (k_{s^6} + P_3^4(k)) + (k^2 k_{s^4} + P^2_5(k))
  = k_{s^6} + P_3^4(k) + P_5^2(k)
\\
\partial_t k_s
 &= k_{s^7} + P_3^5(k) + P_5^3(k) + kk_{s}\SK
  = k_{s^7} + P_3^5(k) + P_5^3(k)
\intertext{and by induction}
\partial_t k_{s^l}
 &=
    k_{s^{l+6}} + P_3^{4+l}(k) + P_5^{2+l}(k)
\end{align*}
The inductive step follows (given the inductive hypothesis) by noticing that
\begin{align*}
 \partial_t k_{s^{l+1}}
&= \partial_s(\partial_tk_{s^l}) + k\SK k_{s^{l+1}}
\\
&= k_{s^{l+7}} + P_3^{5+l}(k) + P_5^{3+l}(k) + (kk_{s^{l+1}}(k_{s^4} + P_3^2(k)))
\\
&= k_{s^{l+7}} + P_3^{5+l}(k) + P_5^{3+l}(k)\,.
\end{align*}
Now we calculate
\begin{align*}
\frac{d}{dt}\int_\gamma (k_{s^l})^2\,ds
 &= 2\int_\gamma
        k_{s^l}(k_{s^{l+6}} + P_3^{4+l}(k) + P_5^{2+l}(k))\,ds
\\&\qquad
   + \int_\gamma k_{s^l}^2k(k_{s^4} + P^2_3(k)))\,ds
\\
 &= -2\int_\gamma
        k_{s^{l+3}}^2\,ds
 + 2\int_\gamma
        k_{s^l}(P_3^{4+l}(k) + P_5^{2+l}(k))\,ds
\,.
\end{align*}
We now apply Proposition 2.5 of Dziuk-Kuwert-Sch\"atzle \cite{DKS} in combination with the length and curvature bounds derived above, yielding the estimates
\[
\int_\gamma k_{s^l}P_3^{4+l}(k)\,ds
\le
	\delta\vn{k_{s^{l+3} }}_2^2 + c(\delta, T, \omega, E[\gamma_0], \vn{k}_\infty)
\]
and
\[
\int_\gamma k_{s^l}P_5^{2+l}(k)\,ds
\le
	\delta\vn{k_{s^{l+3} }}_2^2 + c(\delta, T, \omega, E[\gamma_0], \vn{k}_\infty)
\,.
\]
Choosing $\delta = 1/4$ we find
\begin{align}
\label{theest}
\frac{d}{dt}\int_\gamma k_{s^l}^2\,ds
+ \int_\gamma k_{s^{l+3}}^2\,ds
\le C\,.
\end{align}
This implies that $\vn{k_{s^l}}_2^2$ is uniformly bounded on all bounded time
intervals, yielding the result.
\end{proof}

Now we have global existence by a standard argument.

\begin{thm}
\label{TMtinf}
The ideal curve flow $\gamma:\S\times[0,T)\rightarrow\R^2$ with $\gamma_0$ as initial data exists for all time ($T=\infty$).
\end{thm}
\begin{proof}
Let us suppose $T<\infty$ (and $T$ is maximal).
In this proof we use $C(l)$ to denote a constant only depending on $l$, $\gamma_0$, and $T$.
We use the same symbol $C(l)$ to denote possibly different constants throughout the proof.

The main tool is Theorem \ref{curvatureestimates}.
This implies for all $l\in\N$,
\[
\vn{\partial_s^l\SK}_\infty \le C(l)
\,.
\]
Set $v = |\partial_u\gamma|$ where $u$ is the
initial space parameter before reparametrisation by arc-length.
From the evolution equation $v$ satisfies
\[
\partial_tv = -k\SK v
\]
so that $v$ is uniformly bounded from above and below on any bounded time interval.
We observe that for any function $\phi:\S\rightarrow\R$ we have
\begin{equation}
\label{convertform}
\partial_u^l\phi = v^l\partial_s^l\phi + P^l(v, \ldots, \partial_u^{l-1}v, \phi, \ldots, \partial_s^{l-1}\phi)
\end{equation}
where $P^l$ is a polynomial.
Using this for $\phi = k\SK$, we see that
\[
\vn{\partial_u^l(k\SK)}_\infty \le C(l)
\,.
\]
By differentiating the ODE for $v$ we have for $\psi_l = \partial_u^lv$
\[
\partial_t\psi_l + k\SK\psi_l \le C(l)
\,,
\]
which gives
\[
\vn{\psi_l}_\infty \le C(l)
\,.
\]
The evolution equation $\partial_t\gamma = \SK\nu$, and the equations $|\partial_s\gamma| = 1$, $\partial_s^2\gamma = k\nu$, imply
\[
\vn{\partial_s^l\gamma}_\infty \le C(l)
\]
where $l\in\N_0$.
These uniform estimates imply $\gamma$ extends smoothly to $\S\times[0,T]$, and by short-time existence beyond $T$, contradicting the maximality of $T$.
\end{proof}

\begin{proof}[Proof of Theorem \ref{TMglobal}.]
Theorem \ref{TMtinf} is the $T=\infty$ part of Theorem \ref{TMglobal}.
The remaining part of Theorem \ref{TMglobal} that is to be established is the convergence statement under the assumption that length is bounded.

Intuitively, if the length is uniformly bounded, the smallness assumption
required to use the proof of Theorem \ref{TMmain} as-is should be satisfied
after waiting a sufficient amount of time, and then we can apply the proof of Theorem \ref{TMmain} to the flow starting after this waiting time, in order to conclude Theorem \ref{TMglobal}.
Clearly, since $\vn{\SK}_2^2\in L^1([0,\infty))$, we have convergence along a subsequence to an $\omega$-circle, and so there does exist such a waiting time.
To conclude the proof, we give an estimate for this.

Our estimate exploits the rigidity proof from Section \ref{Srigid}.
We claim that
\begin{equation}
\label{stabilityeqn}
E \leq CL^{3/2}\vn{\SK}_2
\,,
\end{equation}
where $C = C(\omega)$.

To prove \eqref{stabilityeqn}, let $M$ and $N$ be as in Section \ref{Srigid}, and set
\[
Q=(M+iN)e^{-i\theta}
\]
where $\theta$ is the angle of the unit tangent.
Then $Q' = {\mathcal K}e^{-i\theta}$.
That gives
\[
M+iN = Q(x'+i y')\,.
\]
Integrating, we find (with $\bar Q$ the average of $Q$)
\begin{align*}
-\frac32i{E}
	&= \int_\gamma (M+iN)\,ds
\\
	&= \int_\gamma(Q-\bar Q)(x'+i y')\,ds+\bar Q\int_\gamma(x'+i y')\,ds
\end{align*}
which implies
\begin{align*}
E	&\leq CL\vn{Q-\bar Q}_\infty
\\
	&\leq CL\int_\gamma |{\mathcal K}|\,ds
\\
	&\leq CL^{3/2}\|{\mathcal K}\|_{2}
\,,
\end{align*}
as required.

The estimate \eqref{stabilityeqn} implies controlled decay of $E$, since (recall length is uniformly bounded here) then
\begin{align*}
\frac{d}{dt}\int_\gamma k_s^2\,ds
 &= -2\int_\gamma |\SK|^2\,ds
\\
 &\le -2CL_0^{-3/2}\bigg(\int_\gamma k_s^2\,ds\bigg)^2\,,
\end{align*}
implying that 
\[
E(t) \le \frac{E[\gamma_0]}{1+2CL_0^{-3/2}E[\gamma_0]\,t}
\,.
\]
Therefore we can estimate the waiting time $t_0$ by (for example)
\[
t_0 \le \frac{L_0^{3/2}}{2CE[\gamma_0]}\bigg(\frac{E[\gamma_0]}{\varepsilon_2}\bigg) = 
        \frac{L_0^{3/2}}{2C\varepsilon_2}
\,.
\]
This finishes the proof.
\end{proof}


\section{Stability of \texorpdfstring{$\omega$}{omega}-circles}
\label{Sstability}
   
Our primary goal in this section is to use the hypothesis
\begin{equation}
\label{smallenergy}
\tag{E}
L^3(t)\vn{k_s}_2^2(t) < \varepsilon
\end{equation}
at $t=0$ to control evolving length.
First, we prove some preparatory estimates.

\begin{lem}
Let $\SP = \int_\gamma k_{s^4}^2\,ds$ and
\[
\SK_0 = k_{s^4} + \bigg(\frac{2\pi\omega}L\bigg)^2k_{ss}
\,.
\]
Then there exists a constant $C_\omega$ depending only on $\omega$ such that
\[
	\int_\gamma \SK_0^2\,ds \ge C_\omega\SP - 4^5\omega^8\pi^8L^{-3}E^2
\,.
\]
\label{LMskp}
\end{lem}
\begin{proof}
Consider the Fourier series for $k$:
\[
k = \sum_p a_p\exp\bigg(i\frac{2\pi}Lps\bigg)
\,.
\]
Then
\begin{align*}
\SK_0 &= \sum_p a_p\bigg[\bigg(\frac{4\pi^2}{L^2}p^2\bigg)^2 - \frac{4\pi^2}{L^2}\omega^2\frac{4\pi^2}{L^2}p^2\bigg]\exp\bigg(i\frac{2\pi}Lps\bigg)
\\
&= \sum_p a_p\bigg(\frac{4\pi^2}{L^2}\bigg)^2p^2(p^2-\omega^2)\exp\bigg(i\frac{2\pi}Lps\bigg)
\,.
\end{align*}
This implies
\begin{align*}
\int_\gamma \SK_0^2\,ds &= \sum_p|a_p|^2\bigg(\frac{4\pi^2}{L^2}\bigg)^4p^4(p^2-\omega^2)^2L
\,.
\end{align*}
We calculate 
\[
a_{\pm \omega} = \frac1{{L}}\int_\gamma k\exp\bigg(\pm i\frac{2\pi}{L}\omega s\bigg)\,ds
          = \frac1{{L}}\int_\gamma \bigg(k-\frac{2\pi\omega}{L}\bigg) \exp\bigg(\pm i\frac{2\pi}{L}\omega s\bigg)\,ds
\,.
\]
This implies
\begin{align*}
|a_{\pm \omega}|
 &\le \frac1{{L}}\int_\gamma \bigg|k-\frac{2\pi\omega}{L}\bigg|\,\bigg| \exp\bigg(\pm i\frac{2\pi}{L}\omega s\bigg) - \exp\bigg(\pm i\theta\bigg) \bigg|\,ds
 \\&\qquad   + \frac1{{L}}\bigg|\int_\gamma \bigg(k-\frac{2\pi\omega}{L}\bigg) \exp(\pm i\theta)\,ds\bigg|
\,.
\end{align*}
In the above we have again used $\theta$ to denote the tangential angle.

Noting that
\[
\int_\gamma \bigg(k-\frac{2\pi\omega}{L}\bigg) \exp(\pm i\theta)\,ds
= \int_\gamma \theta_s \exp(\pm i\theta)\,ds - \frac{2\pi\omega}{L}\int_\gamma \tau\,ds
= 0
\]
and
\begin{align*}
\bigg| \frac{2\pi}{L}\omega s - \theta\bigg|
 &\le \int_0^s\bigg|
		\frac{d}{ds}\bigg(\frac{2\pi}{L}\omega s - \theta\bigg)
	\bigg|\,ds
\\&\le \int_0^s\bigg|\frac{2\pi}{L}\omega - k
	\bigg|\,ds
 \le L\int_\gamma |k_s|\,ds
 \le \sqrt2 L^\frac32 E^\frac12
\end{align*}
we obtain
\begin{align*}
|a_{\pm \omega}| 
 &\le \frac1{{L}}\int_\gamma \bigg|k-\frac{2\pi\omega}{L}\bigg|\,\bigg| \exp\bigg(\pm i\frac{2\pi}{L}\omega s\bigg) - \exp\bigg(\pm i\theta\bigg) \bigg|\,ds
\\
 &\le \frac1{{L}}\int_\gamma |k_s|\,ds\int_\gamma \bigg| \bigg(\pm i\frac{2\pi}{L}\omega s\bigg) - \bigg(\pm i\theta\bigg) \bigg|\,ds
\\
 &\le L\bigg(\int_\gamma |k_s|\,ds\bigg)^2
\\
 &\le 2L^2E
\,.
\end{align*}
Now
\[
\SP = \sum_p \bigg(\frac{4\pi^2}{L^2}\bigg)^4p^8|a_p|^2L\,,
\]
so we have
\begin{align*}
\int_\gamma \SK_0^2\,ds
 &= \sum_p|a_p|^2\bigg(\frac{4\pi^2}{L^2}\bigg)^4p^4(p^2-\omega^2)^2L
\\
 &= \sum_{|p|\ne \omega}|a_p|^2\bigg(\frac{4\pi^2}{L^2}\bigg)^4p^4(p^2-\omega^2)^2L
\\
 &\ge \sum_{|p|\ne \omega,0}|a_p|^2\bigg(\frac{4\pi^2}{L^2}\bigg)^4p^8\bigg(1-\frac{\omega^2}{p^2}\bigg)^2L
\,.
\end{align*}
We define $C_\omega$ by
\[
\bigg(1-\frac{\omega^2}{p^2}\bigg)^2 \ge \min\bigg\{
		\bigg(1-\frac{\omega^2}{(\omega-1)^2}\bigg)^2,
		\bigg(1-\frac{\omega^2}{(\omega+1)^2}\bigg)^2 
	\bigg\}
:= C_\omega
\,.
\]
Then
\begin{align*}
\int_\gamma \SK_0^2\,ds
 &\ge C_\omega \SP - \bigg(\frac{4\pi^2}{L^2}\bigg)^4\omega^8\bigg(|a_{+\omega}|^2 + |a_{-\omega}|^2\bigg)L
\\
 &\ge C_\omega \SP - 4\omega^8\bigg(\frac{4\pi^2}{L^2}\bigg)^4L^5E^2
\\
 &\ge C_\omega \SP - 4^5\omega^8\pi^8L^{-3}E^2\,,
\end{align*}
as required.
\end{proof}

\begin{prop}
\label{propcoolest}
Let $\gamma_0:\S\rightarrow\R^2$ be a smooth immersed curve.
Then there exist universal constants $\hat C_\omega$ and $\varepsilon_0 > 0$ depending only on $\omega$ such that
\[
	(L^3E)[\gamma] < \varepsilon_0 \quad\Longrightarrow\quad	\int_\gamma \SK^2\,ds \ge \hat C_\omega L^{-6}E
\,.
\]
\end{prop}
\begin{proof}
Using Lemma \ref{LMskp}:
\begin{align}
	\int_\gamma \SK^2\,ds
	&= \int_\gamma \SK_0^2\,ds
		+ 2\int_\gamma \SK_0(\SK-\SK_0)\,ds
		+ \int_\gamma (\SK-\SK_0)^2\,ds
	\notag	\\
	&\ge \frac12\int_\gamma \SK_0^2\,ds - \int_\gamma (\SK-\SK_0)^2\,ds
	\notag	\\
	&\ge 
	     \frac{C_\omega}{2}\SP - 2^9\omega^8\pi^8L^{-3}E^2
		- \int_\gamma (\SK-\SK_0)^2\,ds
		\,.
\label{esti2}
\end{align}
Now $\SK - \SK_0 = (k^2 - (\frac{2\pi\omega}{L})^2)k_{ss} - \frac12kk_s^2$, so
\begin{align*}
	\int_\gamma (\SK-\SK_0)^2\,ds
	&\le 2\int_\gamma \bigg(k^2 - \bigg(\frac{2\pi\omega}{L}\bigg)^2\bigg)^2k_{ss}^2\,ds
	 + \frac12\int_\gamma k^2k_s^4\,ds
	 \,.
\end{align*}
Now the curvature bound (Lemma \ref{kest}) yields
\begin{align}
	\int_\gamma (\SK-\SK_0)^2\,ds
	&\le 2\bigg(\bigg(\sqrt{2LE} + \frac{2\omega\pi}{L}\bigg) + \bigg(\frac{2\omega\pi}{L}\bigg)\bigg)^2\bigg(\int_\gamma |k_s|\,ds\bigg)^2\int_\gamma k_{ss}^2\,ds
\notag\\&\qquad
	 + \frac12\bigg(\sqrt{2LE} + \frac{2\omega\pi}{L}\bigg)^2\int_\gamma k_s^4\,ds
\notag\\
	&\le 4\sqrt2\Big(\sqrt{2L^3E} + 4\omega\pi\Big)^2\frac{E}{L}\int_\gamma k_{ss}^2\,ds
\notag\\&\qquad
	 + \frac12\Big(\sqrt{2L^3E} + 2\omega\pi\Big)^2L^{-2}\int_\gamma k_s^4\,ds
	 \,.
\label{estimate}
\end{align}
Now the Gagliardo-Nirenberg Sobolev inequality yields universal constants $C_3, C_4$ such that the inequalities
\begin{align}
C_1\int_\gamma k_{ss}^2\,ds &\le C_1C_3\bigg(\int_\gamma k_s^2\,ds\bigg)^\frac23
				\bigg(\int_\gamma k_{s^4}^2\,ds\bigg)^\frac13
			\le C_1C_3(2E)^\frac23\SP^\frac13
\notag\\
			&\le \frac{LC_\omega}{8E}\SP + C(\omega)C_1^\frac32C_3^\frac32L^{-\frac{1}{2}}E^\frac32
\,,\text{ and}
\label{estembedded}\\
C_2\int_\gamma k_s^4\,ds &\le C_2C_4L^2\bigg(\int_\gamma k_s^2\,ds\bigg)^\frac32
				\bigg(\int_\gamma k_{s^4}^2\,ds\bigg)^\frac12
			\le C_2C_4L^2(2E)^\frac32\SP^\frac12
\notag\\
			&\le \frac{L^2C_\omega}{8}\SP + C(\omega)C_2^2C_4^2 E^3L^{2}
\notag
\end{align}
hold.
Using these with $C_1 =  4\sqrt2\Big(\sqrt{2L^3E} + 4\omega\pi\Big)^2 $
and $C_2 =\frac12\Big(\sqrt{2L^3E} + 2\omega\pi\Big)^2$
in combination with \eqref{estimate} above yields
\begin{align*}
\int_\gamma (\SK-\SK_0)^2\,ds &\le
		\frac{C_\omega}{4}\SP
		+ C(\omega)\Big[ ((L^3E)^\frac32 + \omega^3) (L^{-\frac32}E^\frac52) + ((L^3E)^2 + \omega^4) E^3\Big]
\end{align*}
where $C(\omega)$ is a constant depending only on $\omega$.
Plugging this into \eqref{esti2} we find
\begin{align*}
	\int_\gamma \SK^2\,ds
	&\ge 
	     \frac{C_\omega}{2}\SP - 2^9\omega^8\pi^8L^{-3}E^2
		- \int_\gamma (\SK-\SK_0)^2\,ds
\\
	&\ge 
	     \frac{C_\omega}{4}\SP
	     - C(\omega)\Big[ L^{-3}E^2 +  ((L^3E)^\frac32 + \omega^3) (L^{-\frac32}E^\frac52) + ((L^3E)^2 + \omega^4) E^3\Big]
\\
	&\ge \frac{C_\omega}{4}\Big(\frac{4\omega^2\pi^2}{L^2}\Big)^3 E
		- C(\omega)\Big[  L^{-3}E^2 + ((L^3E)^\frac32 + \omega^3) (L^{-\frac32}E^\frac52) + ((L^3E)^2 + \omega^4) E^3\Big]
	\,.
\end{align*}
This implies
\[
L^9\int_\gamma \SK^2\,ds
\ge a(L^3E)
- b(L^3E)^2 
- c(L^3E)^\frac52 
- d(L^3E)^3 
- e(L^3E)^4
- f(L^3E)^5
\]
where $a,b,c,d,e,f$ are universal constants that depend only on $\omega$.
Therefore, for $L^3E$ small enough (depending only on $\omega$), we have
\[
L^9\int_\gamma \SK^2\,ds
\ge \frac{a}{2}L^3E\,,
\]
as required.

\end{proof}

The first consequence of Proposition \ref{propcoolest} is preservation and exponential improvement of the scale-invariant smallness condition.

\begin{prop}
\label{proppres}
There exist absolute constants $\varepsilon_1, \varepsilon_2, C_0$ depending only on $\omega$ (with $0 < \varepsilon_2 \le \varepsilon_1<\varepsilon_0$, $C_0>0$) such that the following statements hold.
Let $\gamma_0:\S\rightarrow\R^2$ be a smooth immersed curve satisfying
\[
(L^3E)[\gamma_0] < \varepsilon_1
\]
and $\gamma:\S\times[0,\infty)\rightarrow\R^2$ be the ideal curve flow with $\gamma_0$ as initial data.
Then
\[
	(L^3E)[\gamma_t] \le (L^3E)[\gamma_0]\,.
\]
Furthermore, if $(L^3E)[\gamma_0] < \varepsilon_2$ then for any $t\in[0,\infty)$ we have the estimate
\[
	\int_0^t L^3\vn{\SK}_2^2\,d\tau \le C_0\Big[(L^3E)[\gamma_0] - (L^3E)[\gamma_t]\Big]\,.
\]
\end{prop}
\begin{proof}
We calculate and estimate to find
\begin{align*}
(L^3E)' &= 
-L^3\int_\gamma \SK^2\,ds
 + 3(L^3E)\frac{(-\vn{k_{ss}}_2^2 + \frac72\vn{kk_s}_2^2)}{L}
\\
&\le 
-L^3\int_\gamma \SK^2\,ds
+ \frac{21}{2}(L^3E)L^{-1}\vn{kk_s}_2^2
\,.
\end{align*}
First, since $\varepsilon_1 < \varepsilon_0$, the hypothesis of Proposition \ref{propcoolest} holds for a maximal time interval $[0,\delta)$.
The curvature estimate (Lemma \ref{kest}) and Proposition \ref{propcoolest} imply that on this time interval
\[
\vn{kk_s}_2^2
\le L^{-2}(\sqrt{2L^3E} + 2\omega\pi)^2\bigg(\frac{2}{\hat C_\omega}L^6\int_\gamma \SK^2\,ds\bigg)
\,.
\]
Combining with our first estimate, we find
\begin{align*}
(L^3E)'
&\le 
L^3\int_\gamma \SK^2\,ds\Big(-1
+ \frac{21}{\hat C_\omega}(L^3E)L^{-3}(\sqrt{2L^3E} + 2\omega\pi)^2L^3\Big)
\\
&\le
L^3\int_\gamma \SK^2\,ds\Big(-1
+ \frac{21}{\hat C_\omega}(L^3E)(\sqrt{2\varepsilon_0} + 2\omega\pi)^2\Big)
\,.
\end{align*}
Therefore by assuming that
\[
(L^3E)(0) < \min\bigg\{\varepsilon_0,
 \frac{\hat C_\omega}{21(\sqrt{2\varepsilon_0} + 2\omega\pi)^2)}
\bigg\} := \varepsilon_1
\]
we see that $(L^3E)' \le 0$, showing that $\delta = T$ and preserving the hypothesis $(L^3E) < \varepsilon_1$
for all time.

To see the second statement, take $(L^3E)(0) < \min\{\varepsilon_0,1/2C\} := \varepsilon_2$, and apply Proposition \ref{propcoolest} one more time, to see
\begin{align*}
(L^3E)'
&\le 
 - \frac{C}{2}L^3\int_\gamma \SK^2\,ds
\le -\frac1{C_0} L^3\vn{\SK}_2^2
\end{align*}
implying the result by integration.
\end{proof}

We are now able to use Proposition \ref{proppres} to establish an a-priori estimate on length, the crucial ingredient needed to obtain our convergence result.

\begin{prop}
\label{lengthestimate}
Let $\gamma_0:\S\rightarrow\R^2$ be a smooth immersed curve satisfying
\[
(L^3E)[\gamma_0] < \varepsilon_2
\]
where $\varepsilon_2$ is as in Proposition \ref{proppres},
and $\gamma:\S\times[0,\infty)\rightarrow\R^2$ be the ideal curve flow with $\gamma_0$ as initial data.
Then
\[
L[\gamma_t] \le L[\gamma_0]\exp\Big(C_1(L^3E)[\gamma_0]\Big)\,,
\]
where $C_1$ is an absolute constant depending only on $\omega$.
\end{prop}
\begin{proof}
We estimate the evolution of length using first the curvature estimate (Lemma \ref{kest}):
\begin{align*}
(\log L)'
 &= -\frac{\vn{k_{ss}}_2^2}{L} + \frac{7}{2L}\vn{kk_s}_2^2
\\
 &\le C(\omega)L^{-3}E
\,.
\end{align*}
Then Proposition \ref{propcoolest} implies
\[
(\log L)' \le C(\omega)L^3\vn{\SK}_2^2
\]
which, after application of the estimate in Proposition \ref{proppres}, yields
\[
\log L[\gamma_t] \le \log L[\gamma_0] + C(\omega)(L^3E)[\gamma_0]\,,
\]
which implies the claimed a-priori estimate for length.
\end{proof}

With a uniform upper bound for length in hand, we are able to conclude exponential decay of the energy, and therefore the scale-invariant energy also.

\begin{cor}
\label{corexpdecayE}
Let $\gamma_0:\S\rightarrow\R^2$ be a smooth immersed curve satisfying
\[
(L^3E)[\gamma_0] < \varepsilon_2
\]
where $\varepsilon_2$ is as in Proposition \ref{proppres},
and $\gamma:\S\times[0,\infty)\rightarrow\R^2$ be the ideal curve flow with $\gamma_0$ as initial data.
Then there exists a universal constant $C_2$ (depending only on $\omega$ and the upper bound for length) such that
\[
	\int_\gamma \SK^2\,ds \ge C_2E
\,,
\]
and in particular
\[
E[\gamma_t] \le E[\gamma_0]e^{-C_2t}
\,.
\]
\end{cor}
\begin{proof}
Proposition \ref{propcoolest} implies
\[
E'[\gamma_t] = -\int_\gamma \SK^2\,ds \le -C_2E[\gamma_t]
\]
from which the claim immediately follows.
\end{proof}

We may now use our estimates to bound length uniformly from below.

\begin{lem}
\label{lengthfrombelow}
Let $\gamma_0:\S\rightarrow\R^2$ be a smooth immersed curve satisfying
\[
(L^3E)[\gamma_0] < \varepsilon_2
\]
where $\varepsilon_2$ is as in Proposition \ref{proppres},
and $\gamma:\S\times[0,\infty)\rightarrow\R^2$ be the ideal curve flow with $\gamma_0$ as initial data.
Then
\[
L[\gamma_t] \ge L[\gamma_0]\exp\Big(-C_3\Big)\,,
\]
where $C_3$ is an absolute constant depending only on $\omega$, $E[\gamma_0]$ and $L[\gamma_0]$.
\end{lem}
\begin{proof}
We calculate
\begin{equation}
\log \Big(\frac{1}{L}\Big)'
	= \frac{\vn{k_{ss}}_2^2}{L} - \frac7{2L}\vn{kk_s}_2^2
	\le L^{-1}\vn{k_{ss}}_2^2
	\,.
\label{eqlog1L}
\end{equation}
Now invoking estimate \eqref{estembedded} and its successors in the proof of Proposition \ref{propcoolest}, we find
\begin{equation*}
L^{-1}\int_\gamma k_{ss}^2\,ds
 \le \frac{L^3}{16\omega^4\pi^4}\SP
\end{equation*}
and
\begin{align*}
C_\omega\SP
 &\le  C(\omega)\bigg(
	\int_\gamma \SK^2\,ds + L^{-3}E + \int_\gamma (\SK-\SK_0)^2\,ds
	\bigg)
\\
 &\le  \frac{C_\omega}{2}\SP
	+ C(\omega)\bigg(
	\int_\gamma \SK^2\,ds + L^{-3}E + \Big[ ((L^3E)^\frac32 + \omega^3) (L^{-\frac32}E^\frac52) + ((L^3E)^2 + \omega^4) E^3\Big]
	\bigg)
\end{align*}
so that absorbing yields
\[
L^3\SP \le 
	C(\omega)\bigg(
	L^3\int_\gamma \SK^2\,ds + E + \Big[ ((L^3E)^\frac32 + \omega^3) (L^{\frac32}E^\frac52) + ((L^3E)^2 + \omega^4) L^3E^3\Big]
	\bigg)
\,.
\]
The estimate in Proposition \ref{proppres} and the exponential decay of $E$ (Corollary \ref{corexpdecayE}) implies then that
\[
	\int_0^\infty L^3\SP dt \le C(\omega,E[\gamma_0],L[\gamma_0])
	\,.
\]
We may then integrate \eqref{eqlog1L} to find
\[
	\log\Big(\frac1{L[\gamma_t]}\Big)
	\le 
	\log\Big(\frac1{L[\gamma_0]}\Big)
	+ C_3
\]
which implies the result.
\end{proof}

\begin{rmk}
Note that combined with the discussion around \eqref{standingass}, this
estimate shows that length is bounded uniformly away from zero along any ideal
curve flow.
\end{rmk}

This means we have a uniform length bound.
We now have global existence by Theorem \ref{TMglobal}.
In fact, since length is uniformly bounded from above and below, we have the
following uniform version of Theorem \ref{curvatureestimates}.

\begin{cor}
\label{uniformcurvatureestimates}
Let $\gamma_0:\S\rightarrow\R^2$ be a smooth immersed curve satisfying
\[
(L^3E)[\gamma_0] < \varepsilon_2
\]
where $\varepsilon_2$ is as in Proposition \ref{proppres},
and $\gamma:\S\times[0,\infty)\rightarrow\R^2$ be the ideal curve flow with $\gamma_0$ as initial data.
We have for all $l\in\N_0$,
\[
\vn{k_{s^l}}_\infty \le C(l)
\,,
\]
where $C(l)$ is a constant depending only on $l$ and $\gamma_0$.
\end{cor}
\begin{proof}
First, the estimate for $l=0$ follows from Lemma \ref{kest} and the uniform bound for $L$ from below (Proposition \ref{lengthfrombelow}).

So, let us assume $l>0$.
In the proof of Theorem \ref{curvatureestimates}, we used Proposition 2.5 from Dziuk-Kuwert-Sch\"atzle \cite{DKS}.
The time dependence here was a result of using the exponential-in-time estimate for length, Lemma \ref{lemlength}.
In that proof we also used the standing assumption that length was uniformly bounded from below.
This means we have the estimate \eqref{theest} for all $t\in[0,\infty)$ with a
uniform constant $C$ on the right depending only on $\omega$, $E[\gamma_0]$ and
$L[\gamma_0]$.
That is,
\begin{align*}
\frac{d}{dt}\int_\gamma k_{s^l}^2\,ds
+ \int_\gamma k_{s^{l+3}}^2\,ds
\le C_l\,.
\end{align*}
This implies (with the Poincar\'e inequality)
\begin{equation}
\label{ohanest}
\frac{d}{dt}\int_\gamma k_{s^l}^2\,ds
\le C_l
- \bigg(\frac{L^2}{4\omega^2\pi^2}\bigg)^3\int_\gamma k_{s^{l}}^2\,ds
\,.
\end{equation}
We assume that $C_l>\vn{k_{s^l}}_2^2|_{t=0}$ (if not, replace it by this constant).
Estimate \eqref{ohanest} implies
\begin{equation}
\label{ananest}
\int_\gamma k_{s^l}^2\,ds
 \le 
 \bigg(\frac{4\omega^2\pi^2}{L^2}\bigg)^3C_l
\,.
\end{equation}
To see this, note that \eqref{ananest} is initially true, and if
$\vn{k_{s^l}}_2^2$ grew to ever attain the value on the right hand side of
\eqref{ananest} at $t=t_0$, the estimate \eqref{ohanest} implies that
$(\vn{k_{s^l}}_2^2)'(t_0) \le 0$.
Therefore $\vn{k_{s^l}}_2^2$ can never exceed the value on the right hand side of \eqref{ananest}.

Finally, from \eqref{ananest} we see that
\[
\vn{k_{s^l}}^2_\infty \le L\int_\gamma k_{s^{l+1}}^2\,ds \le C(l,\omega,E[\gamma_0],L[\gamma_0])
\]
as required.
\end{proof}

Global existence and the uniform estimates on length imply that $L(t)$
converges along a subsequence of times $\{t_j\}$, $t_j\rightarrow\infty$.
We can use any of the uniform $L^1$-in-time functions identified earlier to obtain
convergence along possibly a further subsequence to a standard round
$\omega$-circle parametrised by $\gamma_\infty$.
Here $\gamma_\infty$ is not just the limit as $t\rightarrow\infty$ but includes
possible one-off composition with a tangential diffeomorphism.

\begin{rmk}
Although we work toward full convergence, one should be careful to claim
uniqueness of the parametrisation $\gamma_\infty$.
In geometric problems, this is typically false.
For example, taking initial data for the flow to be $\gamma_0(\theta) =
\gamma_\infty(\theta+\frac\pi2)$ (a rotation of the circle $\gamma_\infty$, whatever it
might be) will produce a stationary flow that remains a fixed distance
from $\gamma_\infty$ in all $C^k$-norms.

This is why it is necessary to include at least implicitly a one-time
reparametrisation to obtain full convergence. In our statement, we do this by
saying that we have full convergence for each flow with given initial data to an $\omega$-circle, but
do not pick out a specific parametrisation for this $\omega$-circle.
From a larger perspective, an interesting open question is how to determine in
general properties of the limit from the initial data, for example the centre
of the limit or its radius.
\end{rmk}

Our approach now is to prove that we have convergence of every derivative of the parametrisation by directly integrating and differentiating the evolution equation.
We begin by using exponential decay of $E$ plus uniform estimates for the curvature to obtain, by interpolation, exponential decay of all derivatives of curvature.

\begin{cor}
Let $\gamma_0:\S\rightarrow\R^2$ be a smooth immersed curve satisfying
\[
(L^3E)[\gamma_0] < \varepsilon_2
\]
where $\varepsilon_2$ is as in Proposition \ref{proppres},
and $\gamma:\S\times[0,\infty)\rightarrow\R^2$ be the ideal curve flow with $\gamma_0$ as initial data.
Then for all $l\in\N$,
\[
\vn{k_{s^l}}_\infty \le C(l)e^{-\frac{C_2}{2}t}
\,,
\]
where $C(l)$ is a constant depending only on $l$, $\omega$, $E[\gamma_0]$, $L[\gamma_0]$, and $C_2$ is as in Corollary \ref{corexpdecayE}.
\label{cordecayk}
\end{cor}
\begin{proof}
We estimate
\[
\vn{k_{s^l}}_\infty^2 \le CL^2\vn{k_{s^{l+1}}}_2^2
 \le CL^2\bigg(
		\int_\gamma k_s^2\,ds
	\bigg)^\frac12
	\bigg(
		\int_\gamma k_{s^{2l+1}}^2\,ds
	\bigg)^\frac12
\]
Now the uniform curvature estimates (Corollary \ref{uniformcurvatureestimates}) and uniform length estimates (Proposition \ref{lengthestimate} and Lemma \ref{lengthfrombelow}) imply
\[
\vn{k_{s^l}}_\infty^2 \le C(l)\sqrt{E}
\]
from which the result follows.
\end{proof}

We can use the control above to obtain uniform bounds for all derivatives of the evolving family $\gamma$.

\begin{prop}
Let $\gamma_0:\S\rightarrow\R^2$ be a smooth immersed curve satisfying
\[
(L^3E)[\gamma_0] < \varepsilon_2
\]
where $\varepsilon_2$ is as in Proposition \ref{proppres},
and $\gamma:\S\times[0,\infty)\rightarrow\R^2$ be the ideal curve flow with $\gamma_0$ as initial data.
Then for all $l\in\N_0$,
\[
\vn{\partial_{u^l}\gamma}_\infty \le C'(l) + \sum_{p=0}^l\vn{\partial_{s^l}\gamma_0}_\infty
\,,
\]
where $C'(l)$ is a constant depending only on $l$, $\omega$, $E[\gamma_0]$, $L[\gamma_0]$, and $C_2$ is as in Corollary \ref{corexpdecayE}.
\label{propallderivs}
\end{prop}
\begin{proof}
We claim that for $l\in\N_0$,
\begin{equation}
\label{bestest}
\partial_t\partial_{s^l}\gamma = (\nu+\tau)\sum_{p=0}^l \bigg( P^{4+p}_{1+l-p}(k) + P^{2+p}_{3+l-p}(k)\bigg)
\,.
\end{equation}
We prove this by induction.
First, note that $\SK = P_1^4(k) + P_3^2(k)$ so the equation above holds for $l=0$.
For the inductive step, we differentiate to find
\begin{align*}
\partial_t\partial_{s^{l+1}}\gamma 
	&=
	[\partial_t,\partial_s]\partial_{s^l}\gamma + \partial_s(\partial_t\partial_{s^l}\gamma)
\\
	&= k\SK\partial_{s^{l-1}}\tau
	+ \partial_s\bigg[(\nu+\tau)\sum_{p=0}^l \bigg( P^{4+p}_{1+l-p}(k) + P^{2+p}_{3+l-p}(k)\bigg)\bigg]
\\
	&= k\SK \bigg(\nu\sum_{p+q=l-1} P_p^q(k) + \tau \sum_{p+q=l-2} kP_p^q(k)\bigg)
\\&\qquad
	+ (\nu+\tau)\sum_{p=0}^l \bigg( P^{4+p}_{2+l-p}(k) + P^{2+p}_{4+l-p}(k)\bigg)
\\&\qquad
	+ (\nu+\tau)\sum_{p=0}^l \bigg( P^{4+p+1}_{1+l-p}(k) + P^{2+p+1}_{3+l-p}(k)\bigg)
\\
	&= \nu\sum_{p+q=l-1} \bigg(P_{p+2}^{q+4}(k) + P_{p+4}^{q+2}(k)\bigg) + \tau \sum_{p+q=l-2} \bigg(P_{p+3}^{q+4}(k) + P_{p+5}^{q+2}(k)\bigg)
\\&\qquad
	+ (\nu+\tau)\sum_{p=0}^{l+1} \bigg( P^{4+p}_{1+l-p}(k) + P^{2+p}_{3+l-p}(k)\bigg)
\\
	&= (\nu+\tau)\sum_{p=0}^{l+1} \bigg( P^{4+p}_{1+l-p}(k) + P^{2+p}_{3+l-p}(k)\bigg)
\end{align*}
as required.

Integrating \eqref{bestest} and using Corollary \ref{cordecayk}, we find
\[
\vn{\partial_{s^l}\gamma}_\infty \le \vn{\partial_{s^l}\gamma_0}_\infty + C(l)\int_0^t e^{-\frac{C_2}{2}t'}\,dt'
                                 \le \vn{\partial_{s^l}\gamma_0}_\infty + C'(l)
\,.
\]
Converting arc-length derivatives back to the given derivatives on the parametrisation by \eqref{convertform}, we find
\[
\vn{\partial_{u^l}\gamma}_\infty \le C''(l) + \sum_{p=0}^l\vn{\partial_{s^l}\gamma_0}_\infty\,,
\]
as required.
\end{proof}

We may now deduce full convergence. The full result is as follows.

\begin{thm}
\label{TMmain}
Let $\gamma_0:\S\rightarrow\R^2$ be a smooth immersed curve satisfying
\[
(L^3E)[\gamma_0] < \varepsilon_2
\]
where $\varepsilon_2$ is as in Proposition \ref{proppres},
and $\gamma:\S\times[0,\infty)\rightarrow\R^2$ be the ideal curve flow with $\gamma_0$ as initial data.
Then $\gamma$ converges exponentially fast in the $C^\infty$-topology to a standard round $\omega$-circle.
\end{thm}
\begin{proof}
We use Theorem \ref{TMconvergence} to conclude full convergence of the flow.

We take $(N^n,h) = (\R^2,g^{\R^2})$, $M^m = \S^1$, $f=\gamma$ and $F=\SK$.
Uniform boundedness of $\gamma$ and all its derivatives (Proposition \ref{propallderivs}) implies the first hypothesis of Theorem \ref{TMconvergence} is satisfied.
For the second, we note that (using Corollary \ref{cordecayk} and Proposition \ref{lengthestimate})
\[
\int_\gamma \SK^2\,ds \le Ce^{-C_2t}\,.
\]
This implies
\[
\int_0^T \bigg(\int_\gamma \SK^2\,ds\bigg)^\frac12\,dt \le C\int_0^T e^{-\frac{C_2}{2}t}\,dt \le \hat C
\]
where $\hat C$ is a constant depending only on $\omega$, $E[\gamma_0]$ and $L[\gamma_0]$.

Finally, we consider the third hypothesis.
Uniform boundedness of all derivatives of $\gamma$ (Proposition \ref{propallderivs}) yields that for any sequence $t_j\rightarrow\infty$, the $C^\infty$-norm of $\gamma(t_j,\cdot)$ is uniformly bounded.
We have exponential decay of the energy, and so $E[\gamma(t_j,\cdot)]\rightarrow 0$, which implies that a subsequence $\gamma(t_{j_k},\cdot)$ converges to an $\omega$-circle in the $C^\infty$-topology.
Of course, $\omega$-circles are smooth, and so the third hypothesis is satisfied.

Therefore we apply Theorem \ref{TMconvergence} to conclude full convergence of the flow.
\end{proof}

\begin{rmk}
We note that this theorem implies the following geometric inequality.
Clearly, perturbations of any $\omega$-circle satisfy
$(L^3E)[\gamma] < \varepsilon_2$.
Theorem \ref{TMmain} implies that if $\gamma:\S\rightarrow\R^2$ is in the regular homotopy class of a lemniscate, then
\[
(L^3E)[\gamma] \ge \varepsilon_2
\,.
\]
If this were not the case, then Theorem \ref{TMmain} would imply that the curve $\gamma$ is
diffeomorphic to an $\omega$-circle, which is impossible.
\end{rmk}


\appendix

\section{A convergence result}

In this part of the appendix we briefly prove that bounded flows whose velocity decays in a certain weak sense have unique limits.

\begin{thm}
\label{TMconvergence}
Let $(N^n,h)$ be an $n$-dimensional Riemannian manifold and $M^m$ be an $m$-dimensional manifold with $n>m$.
Suppose $f:M^m\times[0,\infty)\rightarrow N^n$ is a one-parameter family of smooth isometric immersions satisfying
\begin{align*}
\partial_tf = F\,.
\end{align*}
Suppose furthermore that
\begin{itemize}
\item (Uniform bounds) We have the estimates
\[
\int_M |f|^2\,d\mu \le c_1\quad\text{and}\quad
\int_M |H|^2\,|f|^4\,d\mu \le c_2
\]
for time-independent constants $c_1$ and $c_2$.
\item ($L^1$-$L^2$ Velocity) The $L^2$-norm of the velocity is uniformly $L^1$ in time, that is,
\[
\int_0^T \bigg(\int_M |F|^2\,d\mu\bigg)^\frac12\,dt \le c_3
\]
for a constant $c_3$ that does not depend on $T$.
\item (Subconvergence) there exists a smooth immersion $f_\infty:M^m\rightarrow N^n$ and a sequence $\{t_j\}\subset[0,\infty)$, $t_j\rightarrow\infty$, such that $f(\cdot,t_j)\stackbin{C^\infty}{\longrightarrow}f_\infty$.
\end{itemize}
Then $f$ converges to $f_\infty$.
\end{thm}
\begin{proof}
Suppose there exists a sequence $\{s_j\}\subset[0,\infty)$, $s_j\rightarrow\infty$, such that $f(\cdot,s_j)\stackbin{C^\infty}{\longrightarrow}\tilde{f}\ne f_\infty$.
Consider the functional
\[
\BG[f] = \int_M |f-f_\infty|^2\,d\mu\,.
\]
Since $\tilde f$ and $f_\infty$ are smooth, it follows that 
\begin{equation}
\label{eqcontra}
\lim_{j\rightarrow\infty}\BG[f(\cdot,s_j)] \ne 0\,.
\end{equation}
We compute
\begin{align*}
\bigg|\frac{d}{dt}\BG\bigg|
	&\le \bigg|
		2\int_M |F|\,|f-f_\infty|\,\big(1 + |H|\,|f-f_\infty|\big)\,d\mu
		\bigg|
\\
	&\le \vn{F}_2\bigg[\int_M|f-f_\infty|^2\,\big(1 + |H|\,|f-f_\infty|\big)^2\,d\mu\bigg]^\frac12
\\
	&\le  c\vn{F}_2\bigg[
		\int_M|f-f_\infty|^2 + |H|^2\,|f-f_\infty|^4\,d\mu
		\bigg]^\frac12
\\
	&\le  c\vn{F}_2\bigg[
		\vn{f}_2^2 + \vn{f_\infty}_2^2 + \vn{|H|\,|f|^2}_2^2 + \vn{|H|\,|f_\infty|^2}_2^2
		\bigg]^\frac12
\\
	&\le c\vn{F}_2
\,,
\end{align*}
by hypothesis.
This is in contradiction with \eqref{eqcontra}, since then
\[
|\BG[f(\cdot,s_j)] - \BG[f(\cdot,t_j)]| \le c\int_{\min\{s_j,t_j\}}^\infty \vn{F}_2\,dt \longrightarrow 0\,.
\]
Therefore there can not exist such a sequence $\{s_j\}$, and the theorem is proved.
\end{proof}


\begin{bibdiv}
\begin{biblist}

\bib{CharlieThesis}{book}{
	title={The mean curvature flow of submanifolds of high codimension},
	author={Baker, Charles},
	year={2011},
	publisher={PhD Thesis, Australian National University, arXiv:1104.4409}
}

\bib{Chill}{article}{
  title={On the {\L}ojasiewicz--Simon gradient inequality},
  author={Chill, Ralph},
  journal={Journal of Functional Analysis},
  volume={201},
  number={2},
  pages={572--601},
  year={2003},
  publisher={Elsevier}
}

\bib{DPS}{article}{
  author={Dall'Acqua, Anna},
  author={Pozzi, Paola},
  author={Spener, Adrian},
  title={The {\L}ojasiewicz--Simon gradient inequality for open elastic curves},
  journal={Journal of Differential Equations},
  volume={261},
  number={3},
  pages={2168--2209},
  year={2016},
  publisher={Elsevier}
}

\bib{DKS}{article}{
  title={Evolution of Elastic Curves in $\R^n$: Existence and Computation},
  author={Dziuk, Gerhard},
  author={Kuwert, Ernst},
  author={Sch\"atzle, Reiner},
  journal={SIAM journal on Mathematical Analysis},
  volume={33},
  number={5},
  pages={1228--1245},
  year={2002},
  publisher={SIAM}
}

\bib{application}{article}{
  title={3D Euler spirals for 3D curve completion},
  author={Harary, Gur},
  author={Tal, Ayellet},
  booktitle={Proceedings of the twenty-sixth annual symposium on Computational geometry},
  pages={393--402},
  year={2010},
  organization={ACM}
}
\bib{huisken}{article}{
  title={Flow by mean-curvature of convex surfaces into spheres},
  author={Huisken, Gerhard},
  journal={Journal of Differential Geometry},
  volume={20},
  number={1},
  pages={237--266},
  year={1984}
}

\bib{huiskenvpmcf}{article}{
  title={The volume preserving mean curvature flow},
  author={Huisken, Gerhard},
  journal={J. reine angew. Math},
  volume={382},
  number={35-48},
  pages={78},
  year={1987}
}

\bib{mccoyapmcf}{article}{
  title={The surface area preserving mean curvature flow},
  author={McCoy, James},
  journal={Asian Journal of Mathematics},
  volume={7},
  number={1},
  pages={7--30},
  year={2003},
  publisher={International Press of Boston}
}

\bib{para1}{article}{
  title={A sixth order flow of plane curves with boundary conditions},
  author={McCoy, James},
  author={Wheeler, Glen},
  author={Wu, Yuhan},
  journal={arXiv preprint arXiv:1710.09546},
  year={2017}
}

\bib{para2}{article}{
  author={McCoy, James},
  author={Wheeler, Glen},
  author={Wu, Yuhan},
  title={A sixth order curvature flow of plane curves with boundary conditions},
  journal={MATRIX Annals},
  year={2018 (online)}
}

\bib{WP}{article}{
  title={The polyharmonic heat flow of closed plane curves},
  author={Parkins, Scott},
  author={Wheeler, Glen},
  journal={Journal of Mathematical Analysis and Applications},
  volume={439},
  number={2},
  pages={608--633},
  year={2016},
  publisher={Elsevier}
}

\bib{wh6}{article}{
  title={Surface diffusion flow near spheres},
  author={Wheeler, Glen},
  journal={Calculus of Variations and Partial Differential Equations},
  volume={44},
  number={1},
  pages={131--151},
  year={2012},
  publisher={Springer Berlin/Heidelberg}
}

\bib{Wcdf}{article}{
    author={Wheeler, Glen},
    title={On the curve diffusion flow of closed plane curves},
    journal={Annali di Matematica Pura ed Applicata},
    date={2013},
    volume={192},
    pages={931--950},
    }

\end{biblist}
\end{bibdiv}
\end{document}